\documentclass[reqno,12pt,hidelinks]{amsart}
\usepackage{amssymb,amsmath,amsthm,amsxtra,mathtools,mathrsfs,calc,bm,float}
\usepackage{amsmath}
\usepackage[hypertexnames=false]{hyperref}
\usepackage{enumerate}
\usepackage[margin=1in]{geometry}
\usepackage{mleftright}
\mleftright
\usepackage{nccmath}

\usepackage{color}

\def\Z{\mathbb{Z}}
\def\Q{\mathbb{Q}}

\def\R{\mathbb{R}}

\def\N{\mathbb{N}}
\def\C{\mathbb{C}}

\numberwithin{equation}{section}

\DeclareMathOperator{\im}{Im}
\DeclareMathOperator{\re}{Re}
\def\SL{{\rm SL}}
\def\GL{{\rm GL}}

\newcommand{\pfrac}[2]{\Big(\frac{#1}{#2}\Big)}
\newcommand{\ptfrac}[2]{\Big(\tfrac{#1}{#2}\Big)}

\renewcommand{\bar}[1]{\overline{#1}}

\renewcommand{\sl}{\big|}

\DeclareMathOperator{\lcm}{lcm}
\DeclareMathOperator{\CSC}{CSC}

\def\ep{\varepsilon}

\newtheorem{theorem}{Theorem}
\newtheorem{lemma}[theorem]{Lemma}
\newtheorem{corollary}[theorem]{Corollary}
\newtheorem{proposition}[theorem]{Proposition}

\theoremstyle{remark}

\numberwithin{equation}{section}
\numberwithin{theorem}{section}

\usepackage[capitalise]{cleveref}
\crefname{figure}{Figure}{Figures}
\theoremstyle{plain}
\newtheorem*{theorem*}{Theorem}
\crefname{theorem}{Theorem}{Theorems}
\crefname{corollary}{Corollary}{Corollaries}
\newtheorem*{corollary*}{Corollary}
\crefname{corollary*}{Corollary}{Corollaries}
\crefname{lemma}{Lemma}{Lemmas}
\crefname{proposition}{Proposition}{Propositions}
\crefname{conjecture}{Conjecture}{Conjectures}
\newtheorem*{conjecture*}{Conjecture}
\crefname{conjectures*}{Conjecture}{Conjectures}
\crefname{definition}{Definition}{Definitions}
\crefname{hypothesis}{Hypothesis}{Hypotheses}

\usepackage{autonum}

\begin{document}

\title{Zeros of $\GL_2$ $L$-functions on the critical line}
\date{\today}
\author{Nickolas Andersen}
\address{Department of Mathematics, Brigham Young University}
\email{nick@math.byu.edu}
\author{Jesse Thorner}
\address{Department of Mathematics, University of Florida}
\email{jesse.thorner@gmail.com}

\begin{abstract}
We use Levinson's method and the work of Blomer and Harcos on the $\mathrm{GL}_2$ shifted convolution problem to prove that at least 6.96\% of the zeros of the $L$-function of any holomorphic or Maa\ss~ cusp form lie on the critical line.
\end{abstract}

\maketitle

\section{Introduction}

Beginning with the celebrated work of Selberg \cite{Selberg}, the last 80 years have seen numerous results measuring the proportion $\kappa$ of nontrivial zeros of the Riemann zeta function that lie on the critical line $\re(s)=\frac{1}{2}$. To give an incomplete history, after Selberg proved that $\kappa>0$, Levinson \cite{Levinson} proved that $\kappa>\frac{1}{3}$.  Conrey \cite{Conrey} refined Levinson's ideas to  prove that $\kappa>\frac{2}{5}$.  Most recently, Pratt, Robles, Zaharescu, and Zeindler \cite{2018arXiv180210521P} proved that $\kappa>0.41729$, and this seems to be close to the limit of Levinson's method.
	The Riemann hypothesis asserts that all of the nontrivial zeros lie on the line $\re(s)=\frac{1}{2}$, which would imply that $\kappa=1$.

We consider the corresponding problem for $L$-functions of degree two.
Let $f$ be a Hecke newform (holomorphic or Maa\ss~) on $\Gamma_0(N_f)$ with trivial nebentypus.
The generalized Riemann hypothesis predicts that the nontrivial zeros of the associated $L$-function $L(s,f)$ all lie on the critical line.
Hafner \cite{Hafner1,Hafner2} modified  Selberg's techniques to prove that for $f$ of level $N_f=1$, the proportion $\kappa_f$ of such zeros is positive.  
Bernard \cite{bernard_thesis,Bernard} proved that for $f$ holomorphic of squarefree level and trivial nebentypus, we have $\kappa_f > 0.0297$; under the generalized Ramanujan conjecture for Hecke--Maa\ss~ forms, this improves to $\kappa_f>0.0693$.
This was recently improved to $\kappa_f > 0.02976$, and to $\kappa_f>0.06938$ under the generalized Ramanujan conjecture \cite{MR3936086}, using the ideas in \cite{2018arXiv180210521P}.
We improve these results as follows.

\begin{theorem}
\label{thm:main_theorem2}
Let $f$ be a holomorphic or Hecke--Maa\ss~ newform on $\Gamma_0(N_f)$ for any $N_f\geq 1$. Then $\kappa_f>0.0696$.  Under the generalized Ramanujan conjecture, we have $\kappa_f>0.0896$.
\end{theorem}

We follow Bernard \cite{Bernard} in applying Levinson's attack on the problem of obtaining a lower bound for $\kappa_f$.  There are two broad areas of improvement.  
The first is a theoretical improvement involving the spectral decomposition of $\GL_2$ shifted convolution sums due to Blomer and Harcos \cite{bh-gafa}, which increases the length of the allowed mollifier (see Theorem \ref{thm:nu} below).
The second is a computational refinement in the process which translates our improved mollifier length into a numerical lower bound for $\kappa_f$.
While our method applies equally well to both holomorphic and Maa\ss~ cusp forms, we have chosen to give most of the details of the proof in the Maa\ss~ case only, as it is more illuminating and only slightly more involved.

Our theoretical improvement lies far into the proof, and several preliminary reductions are nearly identical to the treatments of Levinson's method in \cite{bernard_thesis,Bernard,Young}.  In order to highlight the new contributions here, we will refer the reader to \cite{bernard_thesis,Bernard,Young} for certain lengthy standard calculations which do not pertain to our improvements.

\subsection*{Acknowledgements}

This work began while the first author was funded by NSF grant DMS-1701638 and the second author was a postdoctoral researcher at Stanford University (funded by a NSF Mathematical Sciences Postdoctoral Fellowship).

\section{A smooth mollified second moment}

Let $f$ be a Hecke--Maa\ss~ newform of level $N_f$ and trivial nebentypus. 
Then $f$ is an eigenfunction of the hyperbolic Laplacian $-y^{2}(\partial_{xx}+\partial_{yy})$ with Laplace eigenvalue $\lambda = \frac{1}{4}+r^2$, where $r\in\R$ or $ir\in[0,\frac{1}{2})$.  
We write the Fourier expansion of $f$ at infinity as
\begin{equation}
	f(z) = \sqrt{y}\sum_{n\neq 0}\lambda_f(n)K_{ir}(2\pi|n|y)e^{2\pi i n x},
\end{equation}
arithmetically normalized with $\lambda_f(1)=1$, where $z=x+iy$ and $K_{ir}(y)$ is the $K$-Bessel function.
Then $f$ gives rise to the $L$-function
\begin{equation}
	L(s,f) = \sum_{n=1}^{\infty}\frac{\lambda_f(n)}{n^s}=\prod_p(1-\lambda(p)p^{-s}+\chi(p)p^{-2s})^{-1},
\end{equation}
where $\chi(p)$ is the trivial character modulo $N$.
By the work of Kim and Sarnak \cite[Appendix]{Kim} and Blomer and Brumley \cite{BB1,BB2}, we know that there exists a constant $\theta \in [0,\tfrac{7}{64}]$ such that we have the uniform bounds
\begin{equation}
	|\lambda_f(n)|\leq d(n)n^{\theta},\qquad |\mathrm{Im}(r)|\leq \theta,
\end{equation}
where $d(n)$ is the number of divisors of $n$.
The Ramanujan conjecture asserts that $\theta=0$.
By Rankin-Selberg theory, we have the bound
\begin{equation}
\label{eq:rankin}
	\sum_{X\leq n\leq 2X} |\lambda_f(n)| \ll_{f,\ep} X^{1+\ep}
\end{equation}
for all $\epsilon>0$, which provides us with an average form of the Ramanujan conjecture.

The main object of study in the $\mathrm{GL}_2$ modification of Levinson's argument (following Young's treatment in \cite{Young}) is an asymptotic formula for a smooth mollified second moment of $L(s,f)$. 
Define $\mu_{f}(n)$ by the convolution identity
\begin{equation}
\sum_{d\mid n}\mu_{f}(d)\lambda_{f}(n/d)=\begin{cases}
1&\mbox{if $n=1$,}\\
0&\mbox{if $n>1$.}	
\end{cases}
\end{equation}
Let $T>0$ be a large parameter, and define
\begin{equation}
\label{eqn:mollifier_def}
	\psi(s)=\sum_{n\leq T^{\nu}}\frac{\mu_{f}(n)}{n^{s+\frac{1}{2}-\sigma_0}}P\Big(\frac{\log(T^{\nu}/n)}{\log T^{\nu}}\Big),\qquad \sigma_0=\mfrac{1}{2}-\mfrac{R}{\log T},
\end{equation}
where $R>0$ and $\nu>0$ are parameters to be determined later and $P\in\mathbb{R}[x]$ satisfies $P(0)=0$ and $P(1) = 1$.
Let $Q\in\mathbb{C}[x]$ satisfy $Q(0)=1$, and define
\begin{equation}
	V(s) = Q\Big( -\frac{1}{2\log T} \frac{d}{ds} \Big) L(f,s).
\end{equation}
As a corollary of his main theorem, Bernard \cite[Theorem~5]{Bernard} proved that for any $\nu$ satisfying
\begin{equation}
\label{eq:nu-range}
	0 < \nu < \mfrac{1-2\theta}{4+2\theta},
\end{equation}
we have
\begin{equation} \label{eq:limit-cPQR}
	\lim_{T\to\infty}\frac{1}{T}\int_1^T |V(\sigma_0+it)\psi(\sigma_0+it)|^2 dt = c(P,Q,2R,\nu/2),
\end{equation}
where
\begin{equation}
	c(P,Q,r,\xi) = 1+\frac{1}{\xi}\int_0^1\int_0^1 e^{2rv}|P'(u)Q(v)+\xi P(u)Q'(v)+\xi r P(u)Q(v)|^2 dudv.
\end{equation}
We then have
\begin{align}
\label{eqn:kappa_lower_bound}
\begin{aligned}
\kappa_f \geq\limsup_{T\to\infty}\Big(1-&\frac{1}{2R}\log\Big[\frac{1}{T}\int_1^T |V(\sigma_0+it)\psi(\sigma_0+it)|^2 dt\Big]\Big)\\
&\geq 1-\inf_{P,Q,R}\frac{\log c(P,Q,2R,\nu/2)}{2R}= 1-\inf_{P,Q,R}\frac{\log c(P,Q,R,\nu/2)}{R},
\end{aligned}
\end{align}
which we bound from below using a computer search (see Section \ref{sec:computation} below).

It seems that any significant theoretical improvement that uses Levinson's method comes from extending the permissible range \eqref{eq:nu-range} for $\nu$.
We will briefly recall the main steps in Bernard's argument leading to \eqref{eq:limit-cPQR} until we come to the most difficult part of the argument: estimating a certain shifted convolution sum.
We will then estimate this sum using the ideas of Blomer and Harcos \cite{bh-gafa}.  This results in the following improvement over \eqref{eq:nu-range}.

\begin{theorem} \label{thm:nu}
Equation \ref{eq:limit-cPQR} holds for any $\nu$ in the range
\begin{equation} \label{eq:nu-interval}
	0 <  \nu < \mfrac 14 - \mfrac{\theta}{2}.
\end{equation}
\end{theorem}

\section{From the second moment to shifted convolution sums}

As in Young's short proof of Levinson's theorem, \eqref{eq:limit-cPQR} follows from an asymptotic formula for a smoothed, perturbed second moment integral.
Let $w:\R\to\R$ be a smooth function with compact support in $[T/4,2T]$ satisfying
\begin{equation}
	w^{(j)}(t) \ll \Big(\mfrac{\log T}T\Big)^j \qquad \text{ for each }j\geq 0.
\end{equation}
Let $\alpha,\beta$ be complex numbers satisfying $\alpha,\beta\ll 1/\log T$ and $|\alpha+\beta|\gg 1/\log T$.
We set
\begin{equation}
	I(\alpha,\beta) = \int_{\R}w(t)L(\tfrac{1}{2}+\alpha+it)L(\tfrac{1}{2}+\beta-it)|\psi(\sigma_0+it)|^2 dt.
\end{equation}
Then the main result of this section is the following.

\begin{theorem}
\label{thm:main_theorem}
	For $\nu$ satisfying \eqref{eq:nu-interval} we have
	\begin{equation}
		I(\alpha,\beta) = c(\alpha,\beta)\int_{\R}w(t)dt+O\Big(\frac{T(\log\log T)^4}{\log T}\Big),
	\end{equation}
	where
	\[
	c(\alpha,\beta)=1+\frac{1}{\nu}\cdot\frac{1-T^{-2(\alpha+\beta)}}{(\alpha+\beta)\log T}\cdot\frac{d^2}{dxdy}\Big[T^{-\nu(\beta x+\alpha y)}\int_0^1 P(x+u)P(y+u)du\Big]\Bigg|_{x=y=0}.
	\]
\end{theorem}

\begin{proof}[Theorem \ref{thm:main_theorem} implies Theorem \ref{thm:nu}]
Once the range of $\nu$ is established as claimed, the passage from Theorem \ref{thm:main_theorem} to \eqref{eq:limit-cPQR} is nearly identical to the corresponding argument given in Sections~2 and 3 of \cite{Young} for the Riemann zeta function.
\end{proof}

To prove Theorem \ref{thm:main_theorem}, we begin with a standard approximate functional equation.
\begin{lemma}
	\label{lem:approx_functional}
	Let $G(u)$ be a holomorphic function which is even and decays rapidly in the vertical strip $|\mathrm{Re}(s)|<4$, satisfying $G(0)=1$.  If $X>0$ and $\re(\alpha),\re(\beta)\in[-\frac 12,\frac 12]$, then
	\begin{multline}
		L(\tfrac{1}{2}+\alpha+it,\varphi)L(\tfrac{1}{2}+\beta+it,\varphi)
		\\ =
		\sum_{m,n\geq 1}\frac{\lambda(m)\lambda_f(n)}{m^{\frac{1}{2}+\alpha}n^{\frac{1}{2}+\beta}}\Big(\frac{m}{n}\Big)^{-it}V_{\alpha,\beta}(mn,t)
		+X_{\alpha,\beta}(t)\sum_{m,n\geq 1}\frac{\overline{\lambda(m)}~\overline{\lambda_f(n)}}{m^{\frac{1}{2}-\beta}n^{\frac{1}{2}-\alpha}}\Big(\frac{m}{n}\Big)^{-it}V_{-\beta,-\alpha}(mn,t),
	\end{multline}
	where
	\begin{align}
	V_{\alpha,\beta}(x,t) &= \frac{1}{2\pi i} \int_{(1)} \frac{G(s)}{s} g_{\alpha,\beta}(s,t)x^{-s}\, ds, \\
	g_{\alpha,\beta}(s,t) &= \frac{L_\infty(f, \frac 12+\alpha+s+it)L_\infty(f, \frac 12+\beta+s-it)}{L_\infty(f, \frac 12+\alpha+it)L_\infty(f, \frac 12+\beta-it)}, \\
	X_{\alpha,\beta}(t) &= \frac{L_\infty(\frac 12-\alpha-it)L_\infty(f,\frac 12-\beta+it)}{L_\infty(\frac 12+\alpha+it)L_\infty(\frac 12+\beta-it)}.
	\end{align}
\end{lemma}
\begin{proof}
	This follows from minor changes to the proof of \cite[Theorem 5.3]{IK}.
\end{proof}

Applying the approximate functional equation to $I(\alpha,\beta)$, we find that
\begin{equation}
	I(\alpha,\beta) = \sum_{a,b\leq T^\nu} \frac{\mu_f(a)\mu_f(b)}{\sqrt{ab}} P\Big( \frac{\log(T^\nu/a)}{\log T^\nu} \Big) P\Big( \frac{\log(T^\nu/b)}{\log T^\nu} \Big) \sum_{\pm} \Big( D_{a,b}^{\pm}(\alpha,\beta) + N_{a,b}^{\pm}(\alpha,\beta) \Big),
\end{equation}
where we have split the sum into diagonal terms
\begin{align}
	D_{a,b}^{+}(\alpha,\beta) &= \sum_{am=bn} \frac{\lambda_f(m)\lambda_f(n)}{m^{\frac 12+ \alpha}n^{\frac 12+\beta}} \int_{-\infty}^{\infty} w(t)  V_{\alpha,\beta}(mn,t) \, dt \\
	D_{a,b}^{-}(\alpha,\beta) &= \sum_{am=bn} \frac{\lambda_f(m)\lambda_f(n)}{m^{\frac 12- \alpha}n^{\frac 12-\beta}} \int_{-\infty}^{\infty} w(t)  X_{\alpha,\beta}(t) V_{-\beta,-\alpha}(mn,t) \, dt
\end{align}
and off-diagonal terms
\begin{align}
	N_{a,b}^{+}(\alpha,\beta) &= \sum_{am\neq bn} \frac{\lambda_f(m)\lambda_f(n)}{m^{\frac 12+ \alpha}n^{\frac 12+\beta}} \int_{-\infty}^{\infty} w(t) \pfrac{bn}{am}^{it} V_{\alpha,\beta}(mn,t) \, dt \\
	N_{a,b}^{-}(\alpha,\beta) &= \sum_{am\neq bn} \frac{\lambda_f(m)\lambda_f(n)}{m^{\frac 12- \alpha}n^{\frac 12-\beta}} \int_{-\infty}^{\infty} w(t) \pfrac{bn}{am}^{it} X_{\alpha,\beta}(t) V_{-\beta,-\alpha}(mn,t) \, dt.
\end{align}

\begin{proposition} \label{prop:diag}
For any $\nu\in (0,1)$ and $\alpha,\beta$ as above, we have
	\begin{multline}
		\sum_{a,b\leq T^\nu} \frac{\mu_f(a)\mu_f(b)}{\sqrt{ab}} P\Big( \frac{\log(T^\nu/a)}{\log T^\nu} \Big) P\Big( \frac{\log(T^\nu/b)}{\log T^\nu} \Big) \sum_{\pm} D_{a,b}^{\pm}(\alpha,\beta) \\
		= c(\alpha,\beta)\int_{\R}w(t)dt+O\Big(\frac{T(\log\log T)^4}{\log T}\Big).
	\end{multline}
\end{proposition}
\begin{proof}
	Bernard \cite[Proposition 4]{Bernard} proves this in the holomorphic case, but the proof is identical in the non-holomorphic setting apart from a minor detail in the proof of \cite[Lemma 8]{Bernard}.  In particular, the error term $O(x^{3/5})$ in the first displayed equation in the proof of Lemma~8 is not known in the non-holomorphic case.  One can achieve a power-saving error term for fixed $f$ using a standard contour integral calculation and the convexity bound for $L(s,f\times f)$, so the proof of Lemma 8 holds in the non-holomorphic case after minor adjustments.
\end{proof}

It remains to show that the contribution from the off-diagonal terms is bounded above by the error term in Proposition \ref{prop:diag}.
In the next three sections we will prove the following upper bound for the off-diagonal sums.

\begin{proposition}
\label{prop:off-diag}
Let $\alpha,\beta$ be complex numbers satisfying $\alpha,\beta\ll 1/\log T$ and $|\alpha+\beta|\gg 1/\log T$.
Then for any $a,b\in \N$ we have
\begin{equation} \label{eq:N-a-b-upper-bound}
	N_{a,b}^{\pm}(\alpha,\beta) \ll (ab)^{\frac 12} T^{\frac 12+\theta} (abT)^\ep.
\end{equation}
\end{proposition}

This immediately settles the contribution from the off-diagonal terms.

\begin{corollary}
\label{cor:off_diag}
For $\nu$ satisfying \eqref{eq:nu-interval} we have
	\begin{equation}
		\sum_{a,b\leq T^\nu} \frac{\mu_f(a)\mu_f(b)}{\sqrt{ab}} P\Big( \frac{\log(T^\nu/a)}{\log T^\nu} \Big) P\Big( \frac{\log(T^\nu/b)}{\log T^\nu} \Big) \sum_{\pm} N_{a,b}^{\pm}(\alpha,\beta) \ll T^{1-\ep}.
	\end{equation}
\end{corollary}

\begin{proof}[Proposition \ref{prop:off-diag} implies Corollary \ref{cor:off_diag}]
	First we note that $|\mu_f(a)|\ll \max(1,|\lambda_f(a)|)$.
	We then apply Proposition \ref{prop:off-diag} and bound everything else trivially to obtain
	\begin{multline}
		\sum_{a,b\leq T^\nu} \frac{\mu_f(a)\mu_f(b)}{\sqrt{ab}} P\Big( \frac{\log(T^\nu/a)}{\log T^\nu} \Big) P\Big( \frac{\log(T^\nu/b)}{\log T^\nu} \Big) \sum_{\pm} N_{a,b}^{\pm}(\alpha,\beta) \\
		\ll T^{\frac 12+\theta+\ep} \Big(\sum_{a\leq T^\nu} 1 + \sum_{a\leq T^\nu} |\lambda_f(a)| \Big)^2 \ll T^{\frac 12+\theta + 2\nu +\ep},
	\end{multline}
	where we used the Rankin-Selberg bound \eqref{eq:rankin} in the last step.
	If $\nu \leq \frac14 - \frac\theta 2-\ep$, then the latter expression is bounded by $T^{1-\ep}$.
\end{proof}

\begin{proof}[Proof of Theorem \ref{thm:main_theorem}]
	This follows immediately from Proposition \ref{prop:diag} and Corollary \ref{cor:off_diag}.
\end{proof}

Our main goal for the rest of the paper is to  prove Proposition \ref{prop:off-diag}.

\section{Preliminaries for the proof of Proposition \ref{prop:off-diag}}



Before we begin the proof of Proposition \ref{prop:off-diag}, there are a few simplifications we can make to the off-diagonal sums $N^{\pm}_{a,b}(\alpha,\beta)$. We will work only with $N^+_{a,b}(\alpha,\beta)$ since the corresponding estimates for $N^-_{a,b}(\alpha,\beta)$ are essentially identical in light of the asymptotics for $V_{\alpha,\beta}(x,t)$ and $X_{\alpha,\beta}(t)$ that follow from Stirling's formula.  See \cite[Corollary 5]{Bernard}, for instance.

First, we show that the terms with $mn$ large or $am$ far from $bn$ contribute negligibly to the off-diagonal terms.


\begin{lemma} \label{lem:N-a-b-approx}
Let $a,b,\alpha,\beta$ be as in the statement of Proposition \ref{prop:off-diag} and fix $\delta>0$.  Define
\begin{equation}
	\mathcal M(a,b) = \Big\{(m,n)\in \Z : am\neq bn, mn\leq T^{2+\delta}, \text{ and } \, \mfrac{am}{bn}=1+O(T^{-1+\delta})\Big\}.
\end{equation}
If $\ep>0$, then we have
\begin{equation} \label{eq:N-a-b-approx}
	N_{a,b}^{+}(\alpha,\beta) = \sum_{(m,n) \in \mathcal M(a,b)} \frac{\lambda_f(m)\lambda_f(n)}{m^{\frac 12+ \alpha}n^{\frac 12+\beta}} \int_{-\infty}^{\infty} w(t) \pfrac{bn}{am}^{it} V_{\alpha,\beta}(mn,t) \, dt + O_{\delta}(T^{-1+\ep}),
\end{equation}
\end{lemma}
\begin{proof}
	We proceed just as in \cite[Lemma 3]{Bernard}, though we use \eqref{eq:rankin} in place of the Ramanujan bound for holomorphic newforms.
\end{proof}

Next, we introduce a dyadic partition of unity as follows.
Let $\rho:(0,\infty)\to\R$ be a smooth function, compactly supported in $[1,2]$ such that
\begin{equation}
	\sum_{\ell\in \Z} \rho(2^{-\ell/2} x) = 1.
\end{equation}
We use this partition to write
\begin{equation}
	\sum_{(m,n)\in \mathcal M(a,b)} F(m,n) = \sum_{(k,\ell)\in \mathcal S} \sum_{am\neq bn} F(m,n) \rho\pfrac{am}{A_k}\rho\pfrac{bn}{A_\ell},
\end{equation}
for a suitable set $\mathcal S$, where $A_k = 2^{k/2}T^{1-\delta}$ and $F(m,n)$ denotes the summand in \eqref{eq:N-a-b-approx}.
Here $\delta$ is a fixed small positive number.
As explained in the proof of Lemma~4 of \cite{Bernard}, we have
\begin{equation} \label{eq:S-def}
	\mathcal S = \Big\{ (k,\ell) \in \Z^2 : \tfrac 12T^{1-\delta}\leq A_k\asymp A_\ell \leq 2T^{1+\delta/2}\sqrt{ab} \Big\}
\end{equation}
and
\begin{equation} \label{eq:Hkl-def}
	1\leq |am-bn|\leq H_{k,\ell} := T^{-1+\delta}\sqrt{A_k A_\ell}.
\end{equation}
So we can write the off-diagonal terms as
\begin{multline} \label{eq:Nab-W1-W2}
	N_{a,b}^{+}(\alpha,\beta) = \frac{1}{2\pi i}\int_{\sigma-i\infty}^{\sigma+i\infty} \frac{G(s)}{s}
	\sum_{(k,\ell)\in \mathcal S} \frac{a^{s+\alpha} b^{s+\beta}}{A_k^{2s+\alpha+\beta}} \\ 
	\times \sum_{1\leq |h|\leq H_{k,\ell}} \sum_{am-bn=h} \frac{\lambda_f(m)\lambda_f(n)}{\sqrt{mn}} W_1\pfrac{am}{A_k} W_2 \pfrac{bn}{A_k} \, dsdt + O(T^{-1+\ep})
\end{multline}
for any $\sigma>0$,
where
\begin{equation} \label{eq:W1-W2-def}
	W_1(x) = \frac{\rho(x)}{x^{s+\alpha}}, \qquad W_2(y) = \frac{\rho(y A_k/A_\ell)}{y^{s+\beta}} \int_{-\infty}^\infty w(t) \Big( 1 + \frac{h/A_k}{y} \Big)^{-it}g_{\alpha,\beta}(s,t)\, dt.
\end{equation}
Thus we are led to the averaged shifted convolution problem of bounding the sums
\begin{equation} \label{eq:shift-conv-sum}
	\sum_{1\leq |h|\leq H} \sum_{am-bn=h} \frac{\lambda_f(m)\lambda_f(n)}{\sqrt{mn}} W_1\pfrac{am}{A_k} W_2 \pfrac{bn}{A_k}
\end{equation}
uniformly in the parameters $a,b,H$.
This is the subject of the next section.

\section{Proof of \cref{prop:off-diag}}

At this point, our treatment diverges from that of \cite{Bernard}, where the shifted convolution sum \eqref{eq:shift-conv-sum} is bounded using work of Blomer \cite{Blomer}.  In order to obtain a stronger estimate for \eqref{eq:shift-conv-sum}, we spectrally decompose \eqref{eq:shift-conv-sum} using the elegant ideas of Blomer and Harcos \cite{bh-duke,bh-gafa}.

\subsection{Preliminaries}
Before we state the result of Blomer and Harcos, we recall some background on the spectrum of the hyperbolic Laplacian and fix some notation.  Let $\Gamma\subseteq \SL_2(\Z)$ be a congruence subgroup.
The weight $k$ Petersson inner product is defined as
\begin{equation}
	\langle g,h \rangle = \iint_{\Gamma\backslash \mathbb H} g(z) \overline{h(z)} y^k \frac{dxdy}{y^2}.
\end{equation}
Here $z=x+iy$ and $g,h$ are holomorphic cusp forms of weight $k\geq 2$ or Maa\ss~ cusp forms (of weight $k=0$).
In what follows, we will always assume that each cusp form $g$ (holomorphic or Maa\ss) in a given basis is spectrally normalized so that $\langle g,g \rangle = 1$.
In particular, our newforms will not generally have their first Fourier coefficient equal to $1$.

Let $\mathcal V(M)$ denote the set of spectrally normalized Hecke--Maa\ss~ newforms of level $M$, i.e.~the normalized Maa\ss~ cusp forms on $\Gamma_0(M)$ which are eigenforms for all of the Hecke operators $T_n$ (not just those with $(n,M)=1$).
We also assume that each $g\in \mathcal V(M)$ satisfies $g(-x+iy) = \ep_g g(x+iy)$ with $\ep_g=\pm 1$.
The vector space of Maa\ss~ cusp forms of level $N$ decomposes as
\begin{equation}
	\bigoplus_{M\mid N} \bigoplus_{g\in \mathcal V(M)} \bigoplus_{d\mid \frac NM} g\big|_d \cdot \C,
\end{equation}
where $g\sl_d(z) := g(dz)$.
The first two direct sums are orthogonal, but the innermost direct sum is not orthogonal in general.  Thus we follow \cite{bm-gafa} in choosing an orthonormal basis given by a Gram--Schmidt orthogonalization process:
\begin{equation}
	\bigoplus_{d\mid \frac NM} g\big|_d \cdot \C = \operatorname{span} \bigg\{ g^{(t)} := \sum_{s\mid t} \alpha_{t,s} g\big|_{s} \  : \  t\mid \mfrac NM  \bigg\}.
\end{equation}
The coefficients $\alpha_{t,s}$ are defined in \cite[Eqn. 40]{bh-gafa} (see also \cite[Section 5]{bm-gafa} for an explicit description).
To ease notation, we define the disjoint union
\begin{equation}
	\mathcal U(N) = \bigsqcup_{M\mid N} \mathcal V(M).
\end{equation}
Let $N_g$ denote the unique level $N$ for which the newform $g$ lies in $\mathcal V(N)$.
Each $g\in \mathcal V(N)$ has a Fourier expansion of the form
\begin{equation}
	g(z) = \sqrt y \, \sum_{n\neq 0} \rho_g(n) K_{r_g}(2\pi |n|y) e^{2\pi i n x},
\end{equation}
where $K_\nu(y)$ is the usual $K$-Bessel function and
\begin{equation}
	\lambda_g := \tfrac 14-r_g^2, \qquad r_g\in i\R_{\geq 0} \cup (0,\tfrac 12),
\end{equation}
is the Laplace eigenvalue of $g$, normalized by
\begin{equation}
	(\Delta_0 + \lambda_g) g := \Big( y^2\Big( \mfrac{\partial^2}{\partial x^2} + \mfrac{\partial^2}{\partial y^2} \Big) + \lambda_g \Big) g = 0.
\end{equation}
Furthermore, the coefficients satisfy
\begin{equation}
	\rho_g(n) = \rho_g(1) \lambda_g(n) \qquad \text{ for all }n \geq 1,
\end{equation}
where $\lambda_g(n)$ is the eigenvalue of the Hecke operator $T_n$.
The Hecke eigenvalues of the functions $g^{(t)}$ are given by
\begin{equation}
	\lambda_g^{(t)}(n) := \sum_{s\mid (t,n)} \alpha_{t,s} \sqrt s \, \lambda_g(n/s).
\end{equation}
From \cite[Remark~6]{bh-gafa}, when $N_g$ is squarefree, we have the bound
\begin{equation}
	\alpha_{t,s} \ll t^\ep (t/s)^{\theta-1/2}.
\end{equation}
This also holds when $N_g$ is not squarefree by \cite[Eqn. 5.6]{bm-gafa}.
It follows that we have the pointwise bound
\begin{equation} \label{eq:lambda-j-t-theta-bound}
	\lambda_g^{(t)}(n) \ll n^{\theta+\ep } (t,n)^{1/2-\theta}.
\end{equation}

The setup is similar for the holomorphic cusp forms.
For $\ell\geq 2$, let $\mathcal H_\ell(N)$ denote the set of weight $\ell$ newforms of level dividing $N$.
Each $g\in \mathcal{H}_{\ell}(N)$ has a Fourier expansion of the form
\begin{equation}
	g(z) = \sum_{n\geq 1} \rho_g(n) n^{\frac{\ell-1}{2}} e(nz),
\end{equation}
where 
\begin{equation}
	\rho_g(n) = \lambda_g(n)\rho_g(1),
\end{equation}
and $\lambda_g(n)$ is the $n$-th Hecke eigenvalue of $g$.
As before, let $N_g$ denote the level of $g$.
We apply a similar Gram--Schmidt orthonormalization as above to construct $g^{(t)}$ for each divisor $t$ of $N/N_g$, with coefficients $\lambda_g^{(t)}(n)$.
These satisfy the pointwise bound \eqref{eq:lambda-j-t-theta-bound} with $\theta$ replaced by $0$.
Lastly, the function $y^{\ell/2}g(z)$ is an eigenform of the weight $\ell$ hyperbolic Laplacian
\begin{equation}
	\Delta_\ell := y^2\Big( \mfrac{\partial^2}{\partial x^2} + \mfrac{\partial^2}{\partial y^2} \Big) - i\ell y \mfrac{\partial}{\partial x}
\end{equation}
with eigenvalue
\begin{equation}
	\lambda_g = \tfrac{\ell}{2}\Big(1-\tfrac{\ell}{2}\Big) = \tfrac 14 - \ptfrac{\ell-1}{2}^2 =: \tfrac 14-r_g^2, \qquad r_g>0.
\end{equation}

For detailed background on the Eisenstein series, we refer the reader to \cite{iwaniec-spectral}.
All we need here is that for each singular cusp $\mathfrak a$ of $\Gamma_0(N)$, there is an Eisenstein series $E_{\mathfrak a}(z,s)$ with Fourier expansion
\begin{equation}
	E_{\mathfrak a}\Big(z,\tfrac 12+ir\Big) = \delta_{\mathfrak a=\infty} y^{\frac 12+ir} + \tau_{\mathfrak a}(0,r) y^{\frac 12-ir} + 2\sqrt y \sum_{n\neq 0} \tau_{\mathfrak a}(n,r) K_{ir}(2\pi |n|y) e^{2\pi i n x}.
\end{equation}
Blomer and Harcos \cite[Sections 2.6--2.7]{bh-gafa} describe an explicit linear combination of Eisenstein series whose coefficients exhibit properties similar to those of the Gram--Schmidt bases of cusp forms described above.
Briefly, let $\mathcal E(N)$ denote the set of pairs $\tau=\{\chi,\chi^{-1}\}$, where $\chi$ is a Hecke character (over $\Q$) of conductor $N_\chi$ such that $N_\chi^2 \mid N$.
We write
\begin{equation}
	N_\tau := N_\chi^2 \quad \text{ and } \quad \lambda_\tau := \tfrac 14-r_\tau^2,
\end{equation}
where $r_\tau \in i\R$ is such that $\chi = |\cdot|^{ir_\tau}$ on $\R$.
For each $\tau\in \mathcal E(N)$ and each $t\mid N/N_\tau$, there is an Eisenstein series with coefficients $\lambda_{\tau}^{(t)}(m) = \lambda_{\chi,\chi^{-1}}^{(t)}(m)$ (see \cite[Eqn. 68]{bh-gafa}).
Here we only need that these coefficients satisfy the pointwise bound \cite[Eqn. 48]{bh-gafa}
\begin{equation} \label{eq:lambda-bound-eis}
	\lambda_{\tau}^{(t)}(m) \ll (t,m) m^\ep.
\end{equation}

 \subsection{The Petersson and Kuznetsov formulas}

We require versions of the Petersson and Kuznetsov formulas for the Gram--Schmidt bases chosen above.
In our notation, the Petersson formula (see \cite[Theorem~9.6]{iwaniec-spectral}) takes the form
{\small
\begin{equation} \label{eq:petersson}
	\frac{\Gamma(\ell-1)}{(4\pi)^{\ell-1}} \sum_{g\in \mathcal H_\ell(N)} |\rho_g(1)|^2 \sum_{t\mid N/N_g} \bar\lambda_g^{(t)}(m) \lambda_g^{(t)}(n) = \delta_{mn} + 2\pi i^{\ell} \sum_{\substack{c>0 \\ c\equiv0(N)}} \frac{S(m,n,c)}{c} J_{\ell-1}\pfrac{4\pi\sqrt{mn}}{c},
\end{equation}}%
where $J_{\nu}(x)$ is the $J$-Bessel function and $S(m,n,c)$ is the Kloosterman sum
\begin{equation}
	S(m,n,c) = \sum_{\substack{d\bmod c \\ (c,d)=1}} e\pfrac{m\bar d+nd}{c}.
\end{equation}
To state the Kuznetsov formula, we first require a test function $\phi(z)$ holomorphic in the strip $|\re(z)|<\frac 23$ which satisfies $\phi(z) \ll (1+|z|)^{-2-\ep}$ for some $\ep>0$.
Then for integers $m,n\geq 1$ we have (see \cite[Theorem~9.3]{iwaniec-spectral})
\begin{multline} \label{eq:kuznetsov}
	\sum_{g\in \mathcal U(N)} \frac{|\rho_g(1)|^2}{\cos \pi r_g} \phi(r_g) \sum_{t\mid N/N_g} \bar\lambda_g^{(t)}(m) \lambda_g^{(t)}( n) +  \sum_{\text{cusps }\mathfrak a} \int_{-\infty}^\infty \bar \tau_{\mathfrak a}(m,r)\tau_{\mathfrak a}( n,r) \frac{\phi(ir)}{\cosh \pi r} \, dr \\
	= \delta_{mn} \, \check\phi + \sum_{0<c\equiv 0(N)} \frac{S(m, n,c)}{c} \check\phi\pfrac{4\pi\sqrt{mn}}{c},
\end{multline}
where $\check\phi$ and $\check\phi(x)$ are the integral transforms
\begin{align}
	 \check\phi := \frac{i}\pi \int_{(0)} r\tan(\pi r) \phi(r)\, dr,\qquad \check\phi(x) := -2i \int_{(0)} J_{2r}(x) \frac{r\phi(r)}{\cos\pi r} \, dr.
\end{align}

\subsection{Spectral decomposition of shifted convolution sums}

Following \cite[Eqn. 92]{bh-gafa}, we combine the formulas above into a hybrid Petersson-Kuznetsov formula as follows. 
Let $\varphi$ be a function defined on $\{z\in \C: |\re(z)|<\tfrac 23\} \cup (\Z + \tfrac 12)$ and holomorphic in the interior of that set.
Suppose that $\varphi$ satisfies the decay condition
\begin{equation}
	\varphi(z) \ll_{\varphi,\epsilon} (1+|z|)^{-2-\ep}
\end{equation}
for some $\ep>0$,
and define
\begin{align}
	 \check\varphi &:=\frac{i}\pi \int_{(0)} r\tan(\pi r) \phi(r)\, dr + \sum_{\ell\geq 2\text{ even}} (\ell-1)\varphi\Big(\mfrac{\ell-1}{2}\Big), \\
	 \check\varphi(x) &:= -2i \int_{(0)} J_{2r}(x) \frac{r\phi(r)}{\cos\pi r} \, dr + 2\pi\sum_{\ell\geq 2\text{ even}} (-1)^{\ell/2}(\ell-1)\varphi\pfrac{\ell-1}{2} J_{\ell-1}\pfrac{4\pi\sqrt{mn}}{c}.
\end{align}
To simplify notation, 
let
\begin{equation}
	\mathcal C(N) = \mathcal U(N) \bigcup \Big(\bigcup_{2\leq \ell\equiv 0(2)}\mathcal H_\ell(N)\Big).
\end{equation}
For each $g\in \mathcal C(N)$ define
\begin{equation}
	C_g = 
	\begin{dcases}
		\frac{(4\pi)^{\ell-1}}{\Gamma(\ell)|\rho_g(1)|^2} & \text{ if }g\in \mathcal H_\ell(N), \\
		\frac{\cos\pi r_g}{|\rho_g(1)|^2} & \text{ if }g\in \mathcal U(N).
	\end{dcases}
\end{equation}
We also follow \cite{bh-gafa} in abbreviating the contribution from the Eisenstein series as CSC, since its exact shape is not important for us; all we need is that $\CSC\geq 0$ when $m=n$.
The following formula is obtained by multiplying \eqref{eq:petersson} by $(\ell-1)\varphi(\frac{\ell-1}{2})$, summing over positive even $\ell$, and  adding \eqref{eq:kuznetsov}.
See \cite[Eqn. 92]{bh-gafa} and \cite[Errata, Item 5]{bh-gafa} for a generalization and a discussion of convergence issues.

\begin{proposition}
Let $m,n\geq 1$.
With the above notation and assumptions, we have
\begin{equation}
	\sum_{g\in \mathcal C(N)} C_g^{-1} \varphi(r_f) \sum_{t\mid N/N_g} \bar\lambda_g^{(t)}(m) \lambda_g^{(t)}(n) + \CSC = \delta_{mn} \check\varphi + \sum_{0<c\equiv 0(N)} \frac{S(m,n,c)}{c} \check\varphi\pfrac{4\pi\sqrt{mn}}{c}.
\end{equation}
\end{proposition}


Using the Kuznetsov formula, one can obtain the following large sieve inequality (see the proof of Lemma 6 of \cite{bh-gafa} and Section 5 of \cite{DI}).
\begin{proposition}	
Let $\{a_m\}$ be a sequence of complex numbers and let $M,X\geq 1$. Then
\begin{equation} \label{eq:large-sieve}
	\sum_{\substack{g\in \mathcal C(N) \\ |r_g|\leq X}} \sum_{t\mid N/N_g} \Big| \sum_{|m|\leq M} a_m \lambda_g^{(t)}(m) \Big|^2 \ll (NX^2 + M)(MNX)^\ep \sum_{|m|\leq M}|a_m|^2.
\end{equation}
\end{proposition}	

We turn now to the shifted convolution sum \eqref{eq:shift-conv-sum}, to which we apply the spectral decomposition of Blomer and Harcos \cite{bh-duke,bm-gafa}.
The spectral decomposition is stated in those papers using the language of automorphic forms, and we reproduce it here in the classical language to match the rest of this paper.
To state the theorem, we first define (see also \cite[Eqn. 88]{bh-gafa})
\begin{equation} \label{eq:88}
	\lVert W \rVert_{A^\ell} := \sum_{j\leq k\leq \ell} \Big( \int_{\R^\times} |W^{(j)}(y)|^2 (|y|+|y|^{-1})^k \, \mfrac{dy}{y} \Big)^{1/2}.
\end{equation}

\begin{theorem}[\cite{bm-gafa}, Theorem~2] \label{thm:BH}
Let $f$ be a holomorphic or Maa\ss~ newform of level $N_f$.
Let $a$ and $b$ be nonzero integers and write $N = \lcm(aN_f, bN_f)$. 
For any nonnegative integers $p,q,r$, let $W_1, W_2: \R^\times \to \C$ be arbitrary functions such that $\lVert W_{1} \rVert_{A^\ell}$ and $\lVert W_{2} \rVert_{A^\ell}$ exist for $\ell := 2(8+p+q+2r)$.
Then for each cusp form or Eisenstein series $g$ of level $N_g \mid N$ and for any $t \mid N/N_{g}$ there exists a function $W_{g,t}: \R^\times \to \C$ depending only on $f$, $W_{1}$, $W_{2}$, $g$, and $t$ such that the following two properties hold.
\begin{enumerate}
	\item For $Y>0$ and $h\in \Z\setminus\{0\}$ there is a spectral decomposition 
	\begin{multline}
		\sum_{\substack{am-bn = h}} 
		\frac{\lambda_{f}(m) \lambda_{f}(n)}{\sqrt{mn}} W_1 \Big( \frac{am}{Y} \Big) \bar W_2 \Big( \frac{bn}{Y} \Big)
		\\= \sum_{g \in \mathcal C(N)} \sum_{t \mid N/N_g} \frac{\lambda^{(t)}_{g}(h)}{\sqrt{h}} W_{g,t} \Big(\frac{h}{Y}\Big) + \int_{\tau\in \mathcal E(N)} \sum_{t\mid N/N_\tau} \frac{\lambda_{\tau}^{(t)}(h)}{\sqrt h} W_{\tau,t}\pfrac{h}{Y} \, d\tau.
	\end{multline}
	\item Let $P\in \C[x]$ be a polynomial of degree at most $p$, and let $\mathcal D$ denote the differential operator $\mathcal D:=P(y\partial_{y})$.
	For $y\in \R^\times$ we have
	\begin{multline} 
		\sum_{g\in \mathcal C(N)} \sum_{t \mid N/N_g} (1+|\lambda_{g}|)^{2r} \Big| \mathcal D W_{g, t}(y) \Big|^2 + \int_{\tau\in \mathcal E(N)} \sum_{t\mid N/N_\tau} (1+|\lambda_{\tau}|)^{2r} \Big| \mathcal D W_{\tau, t}(y) \Big|^2 d\tau \\
		 \ll \big|ab\big|^\ep  \lVert W_{1} \rVert_{A^\ell}^2 \lVert W_{2} \rVert_{A^\ell}^2 |y|^{1-2\theta-\ep } \min(1,|y|^{-2q}),
	\end{multline}
	with an implied constant depending only on $f, p,q,r, P$.
\end{enumerate}
\end{theorem}

Remark~12 of \cite{bh-gafa} also gives the $L^1$ bounds
	\begin{equation}  \label{eq:L1}
		\sum_{g\in \mathcal C(N)} \sum_{t \mid N/N_g} (1+|\lambda_{g}|)^{r} | \mathcal D W_{g, t}(y)|
		 \ll |ab|^{\frac 12+\ep}  \lVert W_{1} \rVert_{A^{\ell'}} \lVert W_{2} \rVert_{A^{\ell'}} |y|^{\frac12-\theta-\ep } \min(1,|y|^{-q})
	\end{equation}
and
	\begin{equation}
		\int_{\tau\in \mathcal E(N)} \sum_{t\mid N/N_\tau} (1+|\lambda_{\tau}|)^{r} | \mathcal D W_{\tau, t}(y) | d\tau \label{eq:L1-eis}
		\ll d^{\frac 14}|ab|^{\ep}  \lVert W_{1} \rVert_{A^{\ell'}} \lVert W_{2} \rVert_{A^{\ell'}} |y|^{\frac12-\theta-\ep } \min(1,|y|^{-q})
	\end{equation}
for any $y\in \R^\times$, where $d$ is the largest square divisor of $\lcm(a,b)$ and $\ell' := 2(10+p+q+2r)$.

We estimate $\lVert W_{1}\rVert_{A^j}$ and $\lVert W_{2}\rVert_{A^j}$ for the functions $W_{1}$ and $W_2$ defined in \eqref{eq:W1-W2-def} as follows.

\begin{lemma} \label{lem:W1-W2}
Fix $B\in \N$.
For $W_1$ and $W_2$ defined in \eqref{eq:W1-W2-def}, let $\sigma=\re(s)>0$. 
Then there exists a constant $C=C(B)>0$ such that for any $(k,\ell)\in \mathcal S$ defined in \eqref{eq:S-def} we have
\begin{equation}
	\lVert W_1 \rVert_{A^B} \ll_B (1+|s|)^C \quad \text{ and } \quad \lVert W_2 \rVert_{A^B} \ll_B (1+|s|)^C T^{1+2\sigma} (\log T)^B.
\end{equation}
\end{lemma}

\begin{proof}
	The upper bound for $\lVert W_1 \rVert_{A^B}$ is a straightforward computation involving the definitions \eqref{eq:W1-W2-def} and \eqref{eq:88}.
	For $\lVert W_2 \rVert_{A^B}$ we use \cite[Eqn. 21]{Bernard} (together with the nearby comment that the implied constant is polynomial in $|s|$), which states that
	\begin{equation}
		\frac{\partial^j}{\partial y^j} \Big( \int_\R w(t) \Big(1+\mfrac hy\Big)^{-it} g_{\alpha,\beta}(s,t) \, dt \Big) \ll_{j} (1+|s|)^{C_1} y^{-j} T^{1+2\sigma} (\log T)^j,
	\end{equation}
	for some $C_1>0$.
	We observe that for $(k,\ell)\in \mathcal S$ we have $A_k\asymp A_\ell$, so because the support of $\rho$ is contained in $[1,2]$, we have $y\asymp 1$ whenever $W_2^{(j)}(y)\neq 0$.
	This, together with the Leibniz rule for derivatives, yields the lemma.
\end{proof}

We are now ready to estimate \eqref{eq:shift-conv-sum}.

\begin{proposition} \label{prop:shift-conv}
Let $s\in \C$ with $\sigma=\re(s)>0$ and define the set $\mathcal S$, the number $H=H_{k,\ell}$, and the functions $W_1$ and $W_2$ by \eqref{eq:S-def}, \eqref{eq:Hkl-def}, and \eqref{eq:W1-W2-def}. Then there exists a constant $B>0$ such that for any $(k,\ell)\in \mathcal S$ we have
\begin{align}
\label{eq:shift-conv-sum-2}
\begin{aligned}
	\sum_{1\leq |h|\leq H} \sum_{am-bn=h}& \frac{\lambda_f(m)\lambda_f(n)}{\sqrt{mn}} W_1\pfrac{am}{A_k} W_2 \pfrac{bn}{A_k} \\
	&\ll (ab)^{\frac 14}T^{2\sigma+\theta} \Big((ab)^{\frac 14}T^{\frac 12}+A_k^{\frac 12}\Big) (1+|s|)^B (abA_kT)^\ep.
\end{aligned}
\end{align}
\end{proposition}

\begin{proof}
By Theorem \ref{thm:BH}, the left-hand side of \eqref{eq:shift-conv-sum-2} equals
\begin{equation}
	\sum_{|h|\leq H}\sum_{g \in \mathcal C(N)} \sum_{t \mid N/N_g} \frac{\lambda^{(t)}_{g}(h)}{\sqrt{h}} W_{g,t} \Big(\frac{h}{A_k}\Big) + \sum_{|h|\leq H}\int_{\tau\in \mathcal E(N)} \sum_{t\mid N/N_\tau} \frac{\lambda_{\tau}^{(t)}(h)}{\sqrt h} W_{\tau,t}\pfrac{h}{A_k} \, d\tau,
\end{equation}
where $N = \lcm(aN_f,bN_f) \ll_f ab$. 
For the Eisenstein series contribution, we apply the pointwise bound \eqref{eq:lambda-bound-eis} and the $L^1$ bound \eqref{eq:L1-eis} with $p=q=r=0$ to obtain
\begin{align}
	\sum_{|h|\leq H}\int_{\tau\in \mathcal E(N)} \sum_{t\mid N/N_\tau} \frac{\lambda_{\tau}^{(t)}(h)}{\sqrt h} W_{\tau,t}\pfrac{h}{A_k} \, d\tau
	&\ll \sum_{|h|\leq H} (h,N) |h|^{-\frac 12+\ep} \int_{\tau\in \mathcal E(N)} \sum_{t\mid N/N_\tau} \Big| W_{\tau,t}\pfrac{h}{A_k} \Big| \, d\tau \\
	&\ll (ab)^{\frac 14+\ep} A_k^{-\frac 12+\theta+\ep} \lVert W_1 \rVert_{A^{20}} \lVert W_2 \rVert_{A^{20}} \sum_{|h|\leq H} (h,N) |h|^{-\theta} \\
	&\ll (ab)^{\frac 14} A_k^{\frac 12} T^{-1+\theta} \lVert W_1 \rVert_{A^{20}} \lVert W_2 \rVert_{A^{20}} (abA_kT)^\ep,
\end{align}
where we used $H\asymp A_k T^{-1+\ep}$ in the last line.
It follows from Lemma \ref{lem:W1-W2} that the Eisenstein series contribution is at most
\begin{equation}
	(ab)^{\frac 14} A_k^{\frac 12} T^{2\sigma+\theta} (1+|s|)^{B_1} (abA_kT)^\ep,
\end{equation}
for some $B_1>0$.

We turn our attention to the discrete spectrum.
Let us first consider the contribution from the large eigenvalues, say $|r_g|\geq R\geq 1$.
Using the pointwise bound \eqref{eq:lambda-j-t-theta-bound} for $\lambda_g^{(t)}(h)$ we find that
\begin{equation}
	\sum_{|h|\leq H} \sum_{\substack{g\in \mathcal C(N) \\ |r_g|\geq R}} \sum_{t\mid N/N_g} \frac{\lambda_g^{(t)}(h)}{\sqrt h} W_{g,t} \pfrac{h}{A_k} \ll \sum_{|h|\leq H} |h|^{-\frac 12+\theta+\ep} (h,N)^{\frac 12-\theta} \sum_{\substack{g\in \mathcal C(N) \\ |r_g|\geq R}}\sum_{t\mid N/N_g} \Big|W_{g,t}\pfrac{h}{A_k}\Big|.
\end{equation}
Let $r\geq 1$.
We apply \eqref{eq:L1} with $(p,q,r)=(0,0,r)$ to the inner sums to obtain
\begin{align}
	\sum_{\substack{g\in \mathcal C(N) \\ |r_g|\geq R}}\sum_{t\mid N/N_g} \Big|W_{g,t}\pfrac{h}{A_k}\Big| 
	&\ll R^{-2r} \sum_{g\in \mathcal C(N)}\sum_{t\mid N/N_g} (1+|\lambda_g|)^r \Big|W_{g,t}\pfrac{h}{A_k}\Big|  \\
	&\ll R^{-2r}(ab)^{\frac 12+\ep } \lVert W_1 \rVert_{A^{20+4r}} \lVert W_2 \rVert_{A^{20+4r}} |h|^{\frac 12-\theta-\ep } A_{k}^{-\frac 12+\theta+\ep} \\
	&\ll R^{-2r} (ab)^{\frac 12+\ep}|h|^{\tfrac 12-\theta-\ep}A_k^{-\frac 12+\theta+\ep}(1+|s|)^{B_2} T^{1+2\sigma}(\log T)^{20+4r},
\end{align}
for some $B_2>0$, where we used Lemma \ref{lem:W1-W2} in the last step.
Since $H\asymp A_k T^{-1+\ep}$, it follows that the total contribution from large eigenvalues $|r_g|\geq R$ is at most
\begin{equation} \label{eq:triv-bound-R}
	R^{-2r} (ab)^{\frac 12} A_k^{\frac 12+\theta} T^{2\sigma} (\log T)^{4r} (1+|s|)^{B_2} (abA_kT)^\ep.
\end{equation}

For the remaining eigenvalues we apply Mellin inversion to obtain
\begin{equation}
	\sum_{\substack{g\in \mathcal C(N) \\ |r_g|\leq R}} \sum_{t\mid N/N_g} \sum_{|h|\leq H}\frac{\lambda_g^{(t)}(h)}{\sqrt h} W_{g,t} \pfrac{h}{A_k} = \frac{1}{2\pi i} \int_{(-\omega)} \sum_{\substack{g\in \mathcal C(N) \\ |r_g|\leq R}} \sum_{t\mid N/N_g} \widehat W_{g,t}(u) \sum_{|h|\leq H}\frac{\lambda_g^{(t)}(h)}{\sqrt h} \pfrac{h}{A_k}^{-u} \, du
\end{equation}
for any real $\omega$ satisfying $|\omega|<\frac{1}{2}-\theta$, where $\widehat W_{g,t}(u) = \int_0^\infty W_{g,t}(y) y^{u-1} \, dy$. 
This converges since \eqref{eq:L1} with $(p,q,r) = (p,1,0)$ implies that
\begin{equation} \label{eq:deriv-W-bound}
	|(1+y\partial_y)^p W_{g,t}(y)| \ll_{p,g,t} \min(y^{-\frac12-\theta-\ep },y^{\frac 12-\theta-\ep }).
\end{equation}
By an application of Cauchy-Schwarz, it suffices to estimate
\begin{equation} \label{eq:need-est-1}
	\sum_{\substack{g\in \mathcal C(N) \\ |r_g|\leq R}} \sum_{t\mid N/N_g} |\widehat W_{g,t}(u)|^2
\end{equation}
and
\begin{equation} \label{eq:need-est-2}
	\sum_{\substack{g\in \mathcal C(N) \\ |r_g|\leq R}} \sum_{t\mid N/N_g} \Big| \sum_{|h|\leq H}\frac{\lambda_g^{(t)}(h)}{\sqrt h} \pfrac{h}{A_k}^{-u} \Big|^2.
\end{equation}

By \eqref{eq:deriv-W-bound} and integration by parts we have
\begin{equation}
	\widehat W_{g,t}(u) = \frac{1}{(u-1)^3} \int_0^\infty \left[(1+y\partial_y)^3 W_{g,t}(y)\right] y^{u-1} \, dy.
\end{equation}
Let $\mathcal D=(1+y\partial_y)^3$. 
Then
\begin{gather}
	\sum_{\substack{g\in \mathcal C(N) \\ |r_g|\leq R}}  \sum_{t\mid N/N_g} |\widehat W_{g,t}(u)|^2\ll \frac{1}{(1+|u|)^3} \int_0^{\infty} \int_0^\infty (yz)^{\re(u)-1} \sum_{g,t} |\mathcal DW_{g,t}(y)| |W_{g,t}(z)| \, dydz \\
	 \ll \frac{1}{(1+|u|)^3} \int_0^{\infty} y^{\re(u)-1} \Big(\sum_{g,t} |\mathcal DW_{g,t}(y)|^2\Big)^{\frac 12} dy \int_0^\infty z^{\re(u)-1} \Big(\sum_{g,t} |W_{g,t}(z)|^2 \Big)^{\frac 12} dz.
\end{gather}
Applying Theorem \ref{thm:BH} with $(p,q,r)=(3,1,0)$ and with $(p,q,r)=(0,1,0)$, we find that \eqref{eq:need-est-1} is at most $(1+|u|)^{-3} \lVert W_1 \rVert_{A^{24}}^2 \lVert W_2 \rVert_{A^{24}}^2$.  
To estimate \eqref{eq:need-est-2}, we apply \eqref{eq:large-sieve} to obtain
\begin{equation}
	\sum_{\substack{g\in \mathcal C(N) \\ |r_g|\leq R}} \sum_{t\mid N/N_g} \Big| \sum_{|h|\leq H}\frac{\lambda_g^{(t)}(h)}{\sqrt h} \pfrac{h}{A_k}^{-u} \Big|^2 \ll A_k^{-2\omega} (NR^2+H) (NRH)^\ep \sum_{h\leq H} h^{-1+2\omega}.
\end{equation}
Choosing $\omega>0$ and using that $H\asymp A_k T^{-1+\delta}$, we find that \eqref{eq:need-est-2} is at most
\begin{equation}
	A_k^{-2\omega}H^{2\omega} (abR^2+H)(abRH)^\ep \ll T^{-2\omega} (abR^2+A_kT^{-1})(abRA_kT)^\ep.
\end{equation}
Taking the estimates for \eqref{eq:need-est-1} and \eqref{eq:need-est-2} together, we see (using Lemma \ref{lem:W1-W2}) that for some $B_3>0$, the contribution from the small eigenvalues is bounded above by
\begin{multline}
	\lVert W_1 \rVert_{A^{24}} \lVert W_2 \rVert_{A^{24}} T^{-\omega} ((ab)^{\frac 12}R+A_k^{\frac 12}T^{-\frac 12}) (abRA_kT)^\ep \int_{(-\omega)}(1+|u|)^{-\frac32} \, du \\
	\ll T^{1+2\sigma-\omega}((ab)^{\frac 12}R+A_k^{\frac 12}T^{-\frac 12})(1+|s|)^{B_3}(abRA_kT)^\ep.
\end{multline}

We choose $\omega=\frac 12-\theta-\ep$ (the maximal choice), $R\asymp(A_k^{1/2+\theta})^{1/2r}$, and $r>\frac{1+2\theta}{4\ep}$ so that $R^{-2r} A_k^{1/2+\theta} \asymp 1$ and $R \ll A_k^\ep$.  
Then the total contribution from the continuous and discrete spectra is $\ll(ab)^{1/4}T^{2\sigma+\theta} ((ab)^{1/4}T^{1/2}+A_k^{1/2}) (1+|s|)^B (abA_kT)^\ep$ for some $B>0$.
\end{proof}

\subsection{Proof of Proposition \ref{prop:off-diag}}
By the assumptions on $\alpha,\beta$ we may assume that $|\alpha|,|\beta|<\ep$.
Then, starting with \eqref{eq:Nab-W1-W2}, we apply Proposition \ref{prop:shift-conv}, so that it suffices to estimate
\begin{equation} \label{eq:N+-est}
	(ab)^{\frac 14+\sigma} T^{2\sigma+\theta} (abT)^\ep \sum_{(k,\ell) \in \mathcal S} A_k^{-2\sigma+\ep} ((ab)^{\frac 14}T^{\frac 12}+A_k^{\frac 12}) \int_{(\sigma)} \frac{|G(s)|}{|s|} (1+|s|)^B  \, ds,
\end{equation}
for $\sigma>0$.
(Recall $G(s)$ from Lemma \ref{lem:approx_functional}.)
Since $G(s)$ decays rapidly as $\im(s)\to\pm\infty$, the $s$-integral converges.
Recall that $A_k=2^{k/2}T^{1-\ep}$ and that for $(k,\ell)\in \mathcal S$ we have $T^{1-\ep}\ll A_k\asymp A_\ell\ll T^{1+\ep}\sqrt{ab}$.
Then for the $(k,\ell)$-sum above, we have
\begin{equation}
	\sum_{(k,\ell) \in \mathcal S} A_k^{-2\sigma+\ep} ((ab)^{\frac 14}T^{\frac 12}+A_k^{\frac 12}) \ll ((ab)^{\frac 14}T^{\frac 12} + (ab)^{\frac 14-\sigma}) (abT)^\ep,
\end{equation}
as long as $0<\sigma<\frac 14$.
It follows that \eqref{eq:N+-est} is bounded above by $(ab)^{\frac 12+\sigma}T^{\frac 12+2\sigma+\theta}(abT)^\ep$.  
Taking $\sigma=\ep$, we obtain the bound stated in the proposition.

\section{Proof of Theorem \ref{thm:main_theorem2}}
\label{sec:computation}

Our starting point is the lower bound for $\kappa_f$ given by \eqref{eqn:kappa_lower_bound}. 
Recall the assumption that $Q\in\mathbb{C}[x]$ and $Q(0)=1$.
It follows from Conrey's analysis in  \cite[Section 4]{Conrey} that
\[
\inf_P c(P,Q,R,\nu/2) = \frac{1+e^{2R}|Q(1)|^2}{2}+\frac{\sqrt{AC-\im(B)^2}}{\tanh(\frac{\nu}{2A}\sqrt{AC-\im(B)^2})},
\]
where
{\small\[
A = \int_0^1 e^{2Ry}|Q(y)|^2dy,\quad B = \int_0^1 e^{2Ry}Q(y)(R\bar{Q(y)}+\Bar{Q'(y)})dy,\quad C = \int_0^1 e^{2Ry}|RQ(y)+Q'(y)|^2 dy.
\]}%
For convenience (and with little loss), we introduce the restriction $Q\in\mathbb{R}[x]$, in which case it follows from the above discussion that
\begin{equation} \label{eq:kappa-ABC}
\kappa_f\geq 1-\inf_{\textup{$Q$ real, $R$}}~\frac{1}{R}\log\Big(\frac{1+e^{2R}|Q(1)|^2}{2}+\frac{\sqrt{AC}}{\tanh(\frac{\nu}{2A}\sqrt{AC})}\Big).
\end{equation}
This reduces the problem to finding a polynomial $Q$ (with $Q(0)=1$) and an $R>0$ that maximizes the right-hand side of \eqref{eq:kappa-ABC}.
Following \cite[Section 7]{MR693393}, we consider
\begin{equation}
	Q(x) = 1 + \sum_{n=1}^M h_n \left[ (1-2x)^{2n-1} - 1 \right],
\end{equation}
with $M\geq 1$ and $h_n\in \R$.
To obtain the lower bounds in Theorem \ref{thm:main_theorem}, we take $M=14$ and use the values of $h_n$ and $R$ given in the tables below (depending on how small $\theta$ is).

\begin{figure}[!htbp]
	\scriptsize
	\begin{tabular}{c|c}
		$R$   & $9.582304437529595324225525968614158462696585619536462654863$\\
		\hline
		$h_1$ & $3.2650607637931750587244792559874799768045041255799130557$ \\
		$h_2$ & $-35.69693130713001721891231865087654064687229555965328678$ \\
		$h_3$ & $307.14883295928891589009756204165288064665563409009168052$ \\
		$h_4$ & $-1739.103927081019501913939892903078766924017002429996574$  \\
		$h_5$ & $6512.9325542974150723150270068720999657089209281615395908$ \\
		$h_6$ & $-16198.8443601293961103032391540414314371460860896192784$  \\
		$h_7$ & $25957.21515703246078149929293503676356429786232735906639$ \\
		$h_8$ & $-23196.1136942987736061546457127666178068318784665665983$ \\
		$h_9$ & $1766.822480793037850524223302703619208156578957397032833$ \\
		$h_{10}$ & $23510.5438732118788135370157003618549490819858577637324$ \\
		$h_{11}$ & $-31147.450313477985980949805234797917653137151413430706$ \\
		$h_{12}$ & $20163.1035156568404078601394891882725477808931702413852$ \\
		$h_{13}$ & $-6904.5360619620713355820026759440194356027492039210252$ \\
		$h_{14}$ & $1001.213829782225368635171849616243036610842971507148658$ \\
		\hline
		& $\kappa_f>0.0696$
		\end{tabular}
		\caption{
	Values of $R$ and $h_n$ when $\theta=\frac{7}{64}$.}
	\end{figure}
\begin{figure}[!htbp]
	\scriptsize
	\begin{tabular}{c|c}
		$R$ & $7.65208835094302319781281223375733559244215115656457449206824$  \\
		\hline
		$h_1$ & $2.76591562468250508070995158206560886257952253926602201302$ \\
		$h_2$ & $-24.953344958701578608094342599145795564509677472168223757$ \\
		$h_3$ & $200.979584757620870082162895767337819051600939086950347820$ \\
		$h_4$ & $-1135.7009379004886683791498121437935780154665640397909333$ \\
		$h_5$ & $4400.79389257036604453210634685290016467996411799112455297$ \\
		$h_6$ & $-11685.944299525457982968825745037097501864284460873563400$ \\
		$h_7$ & $21012.4431008707944609811293884869360821464728093593306850$ \\
		$h_8$ & $-24307.727935478566887581509351616632376257812726305311352$ \\
		$h_9$ & $14812.4249653214080776097144795696331108125704065596268115$ \\
		$h_{10}$ & $1611.9036496014506124399090833449500819481038030909555252$ \\
		$h_{11}$ & $-11123.93686807508540451538386504193951998533566471937387$  \\
		$h_{12}$ & $9234.6455282293357101009990009864484443907079856350398614$ \\
		$h_{13}$ & $-3551.076363873755489914586897948635652602188820369924054$  \\
		$h_{14}$ & $553.88321676561034086729538059117184038452850086361067769$ \\
		\hline
		& $\kappa_f>0.0896$
		\end{tabular}
		\caption{
	Values of $R$ and $h_n$ when $\theta=0$.}
	\end{figure}

\bibliographystyle{abbrv}
\bibliography{zeros-paper.bib}

\end{document}